\documentclass{aims}
\usepackage{amsmath}
\usepackage{paralist}
\usepackage{graphics} 
\usepackage{epsfig} 
\usepackage{enumerate}
\usepackage{graphicx}
\usepackage{epstopdf}
\usepackage[colorlinks=true]{hyperref}
\hypersetup{urlcolor=blue, citecolor=red}

\def\currentvolume{X}
 \def\currentissue{X}
  \def\currentyear{200X}
   \def\currentmonth{XX}
    \def\ppages{X--XX}
     \def\DOI{10.3934/xx.xx.xx.xx}

\theoremstyle{plain}
\newtheorem{theorem}{Theorem}[section]

\theoremstyle{plain}

\theoremstyle{plain}
\newtheorem{remark}[theorem]{Remark}

\numberwithin{equation}{section}

\newcommand{\softd}{{\leavevmode\setbox1=\hbox{d}%
          \hbox to 1.05\wd1{d\kern-0.4ex{\char039}\hss}}}
\newcommand{\abs}[1]{\lvert#1\rvert}
\newcommand{\norm}[1]{\lVert#1\rVert}
\newcommand{\D}{\partial}

\newcommand{\dt}{\Delta t}

\newcommand{\mbb}{\mathbb}
\newcommand{\mcal}{\mathcal}

\newcommand{\veps}{\varepsilon}

\newcommand{\Fr}{\mathrm{Fr}}
\newcommand{\Ma}{\mathrm{Ma}}

\def\mm#1
{

\vspace*{-3ex}

\begin{align}#1\end{align}

\vspace*{0ex}

\noindent}

\newcommand{\ba}{\bar{a}}  
\newcommand{\br}{\bar{\rho}}  
\newcommand{\bu}{\bar{u}}  
\newcommand{\Dlt}{\Delta t}

\newcommand{\gm}{\gamma}

\newcommand{\ld}{\lambda}

\title[AP Schemes for singular limits of Euler equations ]
{An Asymptotic Preserving Time Integrator for Low Mach Number Limits of
  the Euler Equations with Gravity}

\author[Arun and Samantaray]{K. R. Arun and S. Samantaray$^*$}


 \email{arun@iisertvm.ac.in}
 \email{sauravsam13@iisertvm.ac.in}

\thanks{$^*$ Corresponding author: S. Samantaray}

\begin{document}
\maketitle


 \centerline{ School of Mathematics}
 \centerline{Indian Institute of Science Education and Research Thiruvananthapuram}
 \centerline{Thiruvananthapuram - 695551, India}

\bigskip

 \centerline{(Communicated by the associate editor name)}

\begin{abstract}
  We consider two distinguished asymptotic limits of the Euler
  equations in a gravitational field, namely the incompressible and
  Boussinesq limits. Both these limits can be obtained as singular
  limits of the Euler equations under appropriate scaling of the Mach
  and Froude numbers. We propose and analyse an asymptotic preserving
  (AP) time discretisation for the numerical approximation of the
  Euler system in these asymptotic regimes. A key step in the
  construction of the AP scheme is a semi-implicit discretisation of the
  fluxes and the source term. The non-stiff convective terms are treated
  explicitly whereas the stiff pressure-gradient and source term are
  implicit. The implicit terms are combined to get a nonlinear
  elliptic equation. We show that the overall scheme is consistent
  with the respective limit system when the Mach number goes to
  zero. A linearised stability analysis confirms the $L^2$-stability
  of the proposed scheme. The results of numerical experiments validate
  the theoretical findings.    
\end{abstract}

\section{Introduction}
The presence of sound/acoustic waves poses a major challenge in
atmospheric and meteorological flow computations due to their fast
characteristic time scales. Hence, in most of the practical
computations, one relies on the so-called `sound-proof' models in
which the sound waves are eliminated. The incompressible equations,
Boussinesq equations, pseudo-incompressible equations, anelastic
equations etc.\ are sound-proof models frequently used in the
literature, to name but a few. The derivation and analysis of
sound-proof models, study of their regimes of validity etc.\ are
topics of active research even today; see, e.g., \cite{Durran} and the
references cited therein for more details. 

A powerful and systematic method to derive a sound-proof model is an 
asymptotic analysis of the Euler equations in which one or more of
the non-dimensional quantities, such as the Mach, Froude or Rossby
numbers, assume the role of limiting parameters \cite{Klein}. However,
from a mathematical point of view, a sound-proof model is often
recognised as a singular limit of the Euler equations under appropriate
scalings. In addition, sound-proof equation systems are typically of
hyperbolic-elliptic in nature, as opposed to the purely hyperbolic
compressible Euler equations. On the other hand, from a numerical
point of view, approximation of singular limits poses several
challenges: stiffness arising from stringent stability requirements,
reduction of order of accuracy due to the presence of limiting
parameters and so on.  

The goal of the present work is to obtain the incompressible and
Boussinesq equations as two distinguished singular limits of the Euler
equations in a gravitational field under appropriate scalings of the
Mach and Froude numbers. We present their numerical resolution via the
so-called asymptotic preserving (AP) methodology. An AP discretisation
for a singularly perturbed problem in general is a one which reduces
to a consistent discretisation of the limit model when the limits of  
perturbation parameters are taken. In addition, the stability
requirements of the discretisation should remain independent of the
perturbation parameters; see \cite{Jin}. A key step in the
construction of our AP scheme is a semi-implicit time discretisation
based on a splitting of the flux and source terms into stiff and non-stiff
terms. We show the asymptotic consistency of the scheme with the
incompressible and Boussinesq limits as the Mach number approaches
zero. As a first step towards the stability of the scheme in the
asymptotic regime, we perform an $L^2$-stability analysis of the
proposed scheme on a linearised model, namely the wave equation
system. The results of our numerical experiments presented here clearly
validate the AP nature of the proposed scheme.

\section{Isentropic Euler System with Gravity and Its Asymptotic 
  Limits} 
We consider the scaled, isentropic compressible Euler equations with
gravity: 
\begin{align}
  \D_t \rho + \nabla \cdot ( \rho u) &= 0, \label{eq:euler_mass} \\
  \D_t (\rho u ) + \nabla \cdot (\rho u \otimes u ) + \frac{\nabla p}
  {\Ma^2} & = -\frac{\rho e_3}{\Fr^2},  \label{eq:euler_mom}
\end{align}
where $\rho>0$ is the density and $u \in \mbb{R}^3$ is the velocity
vector. Here, $\nabla$, $\nabla \cdot$ and $\otimes$ are respectively
the gradient, divergence and tensor product operators and $e_3$ is the
unit vector in the $x_3$-direction. We assume a simplified equation of
state of an isentropic process, therein the pressure is related to
density via $p=P(\rho) = \rho^{\gamma}$, where $\gamma$ is a 
constant. In \eqref{eq:euler_mass}-\eqref{eq:euler_mom}, the
non-dimensional parameters $\Ma$ and $\Fr$ are respectively, the
reference Mach and Froude numbers.

The goal of the present work is the numerical approximation of some
distinguished asymptotic limits of the Euler system
\eqref{eq:euler_mass}-\eqref{eq:euler_mom} which models slow
convection in a highly stratified medium; see, e.g.\ \cite{Durran,
  Klein} for more details. In order to describe these asymptotic
regimes, in the following, we consider two important scalings of $\Ma$
and $\Fr$ in terms of an infinitesimal parameter $\veps$.  
\begin{itemize}
\item $\Ma=\veps$ and $\Fr=1$. In this case, the pressure gradient
  term dominates the gravity term and we obtain the low Mach number
  limit. 
\item $\Ma=\veps$ and $\Fr=\sqrt{\veps}$. In this case, the
  gravitational term is also significant we derive the Boussinesq
  limit. 
\end{itemize}

As a first step towards the derivation of the low Mach and Boussinesq
limits, we expand all the dependent variables using the following
three-term ansatz:   
\begin{align}
  f(t,x) = f_{(0)}(t, x) + \veps f_{(1)} (t, x) + \veps^2 f_{(2)}(t,
  x). \label{eq:ansatz}
\end{align}
We do not intent to provide the details of the derivation, but refer
the interested reader to \cite{Klein} for more details.  

\subsection{Zero Mach Number Limit} 
\label{eq:zero_mach_lim}

We set $\Ma=\veps$ and $\Fr = 1$ in
\eqref{eq:euler_mass}-\eqref{eq:euler_mom} and let $\veps\to0$ to
obtain the zero Mach number limit model:
\begin{align}
  \D_t u_{(0)}+ \nabla \cdot \left( u_{(0)}
  \otimes u_{(0)} \right) + \nabla p_{(2)} &= - e_3,  \label{eq:LM_mom}\\
  \nabla \cdot u_{(0)} &= 0 \label{eq:LM_div}.
\end{align}
The above system \eqref{eq:LM_mom}-\eqref{eq:LM_div} is the standard
incompressible Euler system for the unknowns $u_{(0)}$ and $p_{(2)}$.
\begin{remark}
  Throughout our analysis and the numerical experiments presented in
  this paper, we assume either periodic or wall boundary conditions. As
  a consequence, the leading order density $\rho_{(0)}$ is a constant
  and the leading order velocity $u_{(0)}$ is divergence-free. Therefore,
  both the zero Mach and Boussinesq limits fall in the category of
  `sound-proof' models.  
\end{remark}

\subsection{Boussinesq Limit}
\label{eq:bousinesq_lim}
Now we set $\Ma = \veps$ and $\Fr = \sqrt{\veps}$ in
\eqref{eq:euler_mass}-\eqref{eq:euler_mom}. Letting $\veps \to 0$
yields the Boussinesq model:
\begin{align}
  \D_t u_{(0)} + \nabla \cdot \left( u_{(0)} \otimes u_{(0)} \right) + \nabla
  p_{(2)}  &= - \rho_{(1)} e_3, \label{eq:BO_mom}\\
  \nabla \cdot u_{(0)} & = 0 \label{eq:BO_div}. 
\end{align}
Since the first order density $\rho_{(1)}$ appears in
\eqref{eq:BO_mom}-\eqref{eq:BO_div}, we need a closure relation. Using
the multiscale ansatz \eqref{eq:ansatz} in the equation of state $p =
\rho^{\gamma}$ and using the hydrostatic balance $\nabla
p_{(1)}=-\rho_{(0)}e_3$ gives
\begin{equation}
  \rho_{(1)} = 1-\frac{x_3}{\gamma}.
\end{equation}

\begin{remark}
  It has be noted that both zero Mach and the Boussinesq limit systems
  are hyperbolic-elliptic in nature. 
\end{remark}

\section{Semi-implicit Time Discretisation}
In this section we present the time discretisation of the Euler system
\eqref{eq:euler_mass}-\eqref{eq:euler_mom} based on implicit-explicit
(IMEX) Runge Kutta (RK) schemes. These schemes were originally
designed for stiff ordinary differential equations; see .e.g.\
\cite{Pareschi-Russo} and the references therein.

Let $0 = t^0 < t^1 < \cdots < t^n < t^{n+1} < \cdots$ be an increasing
sequence of times and let $\dt$ be the uniform time-step. Let us
denote by $f^{n} (x)$, the approximation to the value of any function
$f$ at time $t^n$, i.e.\ $f^{n}(x) \sim f(t^n, x)$. 

A first order accurate semi-discrete scheme for the Euler equations
\eqref{eq:euler_mass}-\eqref{eq:euler_mom} is defined as 
\begin{align}
  \frac{\rho^{n+1} - \rho^{n}}{\dt} + \nabla \cdot q^{n+1} &=
                                                                  0, \label{eq:euler_mas_TD}
  \\ 
  \frac{q^{n+1} - q^{n}}{\dt} + \nabla \cdot \left(\frac{ q
  \otimes
  q
  }{\rho}\right)^{n}
  + \frac{\nabla p(\rho^{n+1})}{\veps^2} &= -\frac{\rho^{n+1}}{\veps^{\alpha}}
                                     e_3 \label{eq:euler_mom_TD}. 
\end{align}
Here, $q=\rho u$ denotes the momentum and $\alpha\in \{0,1\}$ is a
parameter so that $\alpha=0$ corresponds to the low Mach limit and
$\alpha=1$ corresponds to the Boussinesq limit. Though the scheme
\eqref{eq:euler_mas_TD}-\eqref{eq:euler_mom_TD} consists of a fully
implicit step \eqref{eq:euler_mas_TD} and a semi-implicit step
\eqref{eq:euler_mom_TD}, its numerical resolution is fairly
simple. Eliminating $q^{n+1}$ between \eqref{eq:euler_mas_TD} and
\eqref{eq:euler_mom_TD} yields the nonlinear elliptic equation:
\begin{align}
  -\frac{\dt^2}{\veps^2}\Delta P(\rho^{n+1}) -
  \frac{\dt^2}{\veps^{\alpha}}\nabla \cdot \left(\rho^{n+1}e_3\right)
  +  \rho^{n+1}= \rho^n - \Phi(\rho^n, u^{n}) \label{eq:elliptic_problem}, 
\end{align}
where the known expression $\Phi$ is given by 
\begin{equation}
  \Phi(\rho^n,u^n) := \dt \nabla \cdot q^n + \dt^2
  \nabla^2:\left(\frac{q \otimes q}{\rho}\right)^n 
\end{equation}
with $:$ denoting the contracted product. Solving the elliptic
equation \eqref{eq:elliptic_problem} yields the updated density
$\rho^{n+1}$. The velocity $u^{n+1}$ can then be updated using
\eqref{eq:euler_mom_TD}, which is now an explicit evaluation. Hence,
the scheme \eqref{eq:euler_mas_TD}-\eqref{eq:euler_mom_TD} consists of 
solving the elliptic equation \eqref{eq:elliptic_problem}, followed by
an explicit evaluation of \eqref{eq:euler_mom_TD}. 

\section{Asymptotic Preserving Property}
A numerical scheme for a singular perturbation problem, such as the
Euler system \eqref{eq:euler_mass}-\eqref{eq:euler_mom}, may
not resolve the existing multiple scales in space and time. In
addition, when the perturbation parameter goes to zero, the scheme may
approximate a completely different set of equations than the actual
limiting systems. An asymptotic preserving (AP) scheme is the one
which is consistent with the limiting set of equations in the singular
limit; see \cite{Jin} for a review of AP schemes. 
\begin{theorem}
  \label{eq:ap_property}
  The time semi-discrete scheme
  \eqref{eq:euler_mas_TD}-\eqref{eq:euler_mom_TD} for $\alpha = 0$ is
  asymptotically consistent with the low Mach number model as
  $\veps\to0$. 
\end{theorem}
\begin{proof}
  First, we apply the same ansatz \eqref{eq:ansatz} for all the dependent
  variables at times $t^{n}$ and $t^{n+1}$ in the semi-discrete scheme
  \eqref{eq:euler_mas_TD}-\eqref{eq:euler_mom_TD} and balance the
  like-powers of $\veps$. The lowest order terms gives $\nabla
  P(\rho^{n+1}_{(0)}) = 0$ and the equation of state $P(\rho) =
  \rho^{\gamma}$ then yields that $\rho^{n+1}_{(0)}$ is
  constant. Therefore, from the mass update \eqref{eq:euler_mas_TD} we
  get 
  \begin{align}
    -\nabla \cdot u^{n+1}_{(0)} = \frac{\rho^{n+1}_{(0)} -
    \rho^n_{(0)}}{\rho^{n+1}_{(0)} \dt}. \label{eq:euler_mass_TD_mod}
  \end{align}

  We integrate the above equation \eqref{eq:euler_mass_TD_mod} over a
  domain $\Omega$ and use Gauss' divergence theorem to obtain:
  \begin{equation}
    -\frac{1}{\abs{\Omega}}\int_{\D\Omega}u^{n+1}_{(0)}\cdot\nu
    d\sigma = \frac{\rho^{n+1}_{(0)} -
      \rho^{n}_{(0)} }{\rho^{n+1}_{(0)}
      \dt}. \label{eq:vel_int}
  \end{equation}
  Hence, the leading order density $\rho_{(0)}$ rises or falls only due to
  compressions or expansions at the boundary. The temporal variations in
  $\rho_{(0)}$ can produce nonzero divergences in the leading order
  velocity $u_{(0)}$. It can be proved that the integral on the left
  hand side of \eqref{eq:vel_int} vanishes under most of the physically
  relevant boundary conditions. In this case, we obtain $\rho^{n+1}_{(0)}
  = \rho^{n}_{(0)}$ and this in turn enforces the divergence constraint
  at $t^{n+1}$ as
  \begin{equation}
    \nabla \cdot u^{n+1}_{(0)} = 0.
    \label{eq:div-u0}
  \end{equation}
  Combining \eqref{eq:div-u0} and the $\mcal{O}(1)$ terms in
  \eqref{eq:euler_mom_TD}, we have the following limiting system:
  \begin{align}
    \frac{u^{n+1}_{(0)} - u^{n}_{(0)}}{\dt} + \nabla \cdot ( u^{n}_{(0)}
    \otimes u^{n}_{(0)}) + p^{n+1}_{(2)} &=
                                           -e_3, \label{eq:euler_mom_TD_limit} \\
    \nabla \cdot u^{n+1}_{(0)} &= 0 \label{eq:div_cond}. 
  \end{align}
  The above system \eqref{eq:euler_mom_TD_limit}-\eqref{eq:div_cond} is
  clearly a consistent discretisation of the low Mach number limit system
  \eqref{eq:LM_mom}-\eqref{eq:LM_div}. 
\end{proof}

\begin{theorem}
  The time semi-discrete scheme \eqref{eq:euler_mas_TD}-\eqref{eq:euler_mom_TD}
  for $\alpha = 1$ is asymptotically consistent with the Boussinesq
  model.
\end{theorem}
\begin{proof}
  The proof is similar to that of Theorem~\ref{eq:ap_property} and
  hence omitted. 
\end{proof}

\section{$L^2$ Stability Analysis of the Semi-discrete Scheme}
\label{sec:L2_stability_semi_disc}
The aim of this section is to present the results of an
$L^2$-stability analysis of the semi-discrete scheme
\eqref{eq:euler_mas_TD}-\eqref{eq:euler_mom_TD}. To this end, we
consider the homogeneous linear wave equation system:
\begin{align}
  \D_t\rho+(\bar{u}\cdot\nabla)\rho+\bar{\rho}\nabla\cdot u &=
                                                            0, \label{eq:wave_mass} \\
  \D_t
  u+(\bar{u}\cdot\nabla)u+\frac{\ba^2}{\br\veps^2}\nabla\rho&=0 \label{eq:wave_mom}   
\end{align}
as a simplified model of the Euler system
\eqref{eq:euler_mass}-\eqref{eq:euler_mom}. Here,
$(\bar{\rho},\bar{u})$ is a linearisation state and $\ba$ is a
linearisation state for the sound velocity. Applying the AP methodology
introduced in \eqref{eq:euler_mas_TD}-\eqref{eq:euler_mom_TD} to
\eqref{eq:wave_mass}-\eqref{eq:wave_mom} yields the semi-discrete
scheme:
\begin{align}
  \frac{\rho^{n+1}-\rho^n}{\dt}+(\bar{u}\cdot\nabla)\rho^n+\bar{\rho}\nabla\cdot
  u^{n+1} &=0, \label{eq:wave_mas_TD} \\
  \frac{u^{n+1}-u^n}{\dt}+(\bar{u}\cdot\nabla)u^n+\frac{\ba^2}{\br\veps^2}\nabla\rho^{n+1}&=0. \label{eq:wave_mom_TD} 
\end{align}
In the following, we use a stability result due to Richtmyer; see e.g.\
\cite{Richtmyer,Buchanan} for details. Note that any difference scheme
of the form $B_1 U^{n+1} = B_2 U^n$, where $B_1, B_2$ are $p \times p$
matrices, independent of $t$ and $x$, and $U^n \in \mbb{R}^p$ is the
approximation to the original solution at time $t^n$, can be reduced
to $\hat{U}^{n+1} = G(\Dlt,\xi) \hat{U}^n$ in the Fourier variable
$\xi$. Here, $G(\Dlt,\xi)$ is the Fourier transform of the matrix
$(B_1)^{-1}B_2$ and is called the amplification matrix. The stability
result due to Richtmyer states that 
\begin{theorem}
  \label{thm:Richtmyer}
  A difference scheme given by $B_1 U^{n+1} = B_2 U^n$ is stable if 
  \begin{enumerate}[(i)]
  \item the elements of $G(0,\xi)$ are bounded for all $\xi \in
    \mbb{L}$, where $\mbb{L}$ is a lattice where $\xi$ varies, 
  \item $\norm{G(0,\xi)} \leq 1$ and
  \item $G(\Dlt,\xi)$ is Lipschitz continuous at $\Dlt = 0$ in the sense that
    \begin{equation*}
      G(\Dlt,\xi) = G(0,\xi) + \mcal{O}(\Dlt) \text{ as } \Dlt \to 0.
    \end{equation*}
  \end{enumerate}
\end{theorem}
Using the above theorem, we have the following stability result. 
\begin{theorem}
  \label{thm:L2_stability_1stOrd_semi_disc}
  The semi-discrete scheme
  \eqref{eq:wave_mas_TD}-\eqref{eq:wave_mom_TD} is $L^2$-stable.  
\end{theorem}
\begin{proof}
  Taking the Fourier transform of
   \eqref{eq:wave_mas_TD}-\eqref{eq:wave_mom_TD} and re-arranging the
   terms gives 
   \begin{equation}
     \label{eq:FT_semi_disc}
     \hat{U}^{n+1} = G(\Dlt,\xi) \hat{U}^n,
   \end{equation}
   where
   \begin{equation}	\label{eq:amp_mat}
     G(\Dlt,\xi) = \gm \begin{pmatrix}
       1			& -i \Dlt \br \xi_1		& -i \Dlt \br \xi_2 \\
       -i \Dlt \ld \xi_1	& 1 + \Dlt^2 \br \ld \xi_2^2	&
       -\Dlt^2 \br \ld \xi_1 \xi_2 \\ 
       -i \Dlt \ld \xi_2	& -\Dlt^2 \br \ld \xi_1 \xi_2	& 1 +
       \Dlt^2 \br \ld \xi_1^2 
     \end{pmatrix},
     \ld = \frac{\ba^2}{\br \veps^2}  \ \text{and}  \ \gm = \frac{1 -
       i \Dlt (\bu \cdot \xi)}{1 + \Dlt^2 \br
       \frac{\ba^2}{\veps^2}|\xi|^2}. 
   \end{equation}
   Now, $G(0,\xi)$ reduces to the $3 \times 3$ identity matrix and
   hence conditions (i) and (ii) of Theorem~\ref{thm:Richtmyer} are
   automatically satisfied. Further, 
   \begin{align}
     \label{eq:amp_diff}
     G(\Dlt,\xi) - G(0,\xi) = \Dlt
     \begin{pmatrix}
       -\frac{\Dlt\br\ld |\xi|^2 + i(\bu\cdot\xi)}{1 + \Dlt^2\br\ld
         |\xi|^2}	& -i\br \xi_1 \gm	& -i\br \xi_2 \gm
       \\ 
     -i\ld \xi_1 \gm	&-\frac{\Dlt\br\ld \xi_1^2 + i(\bu\cdot\xi)(1
       + \Dlt^2\br\ld \xi_2^2)}{1 + \Dlt^2\br\ld |\xi|^2} &
     -\Dlt^2\br\ld\xi_1\xi_2\gm	\\ 
     -i\ld \xi_2 \gm	& -\Dlt^2\br\ld\xi_1\xi_2\gm
     &-\frac{\Dlt\br\ld \xi_2^2 + i(\bu\cdot\xi)(1 + \Dlt^2\br\ld
       \xi_1^2)}{1 + \Dlt^2\br\ld |\xi|^2} 
    \end{pmatrix}.
   \end{align}
   Note that the matrix on the right hand side in \eqref{eq:amp_diff}
   is bounded for every bounded lattice $\mbb{L}$. Hence, by
   Theorem~\ref{thm:Richtmyer}, the semi-discrete scheme
   \eqref{eq:wave_mas_TD}-\eqref{eq:wave_mom_TD} is $L^2$-stable.   
  \end{proof}

  \section{Numerical Experiments}
  \label{sec:num_exp}
  We do not intend to discuss the space discretisation in detail as we
  use employ standard techniques. We use a finite volume approach to
  approximate the semi-discrete scheme
  \eqref{eq:euler_mas_TD}-\eqref{eq:euler_mom_TD}. The explicit flux
  terms are approximated by a Rusanov-type flux whereas the implicit
  terms by simple central differences. The nonlinear elliptic equation
  \eqref{eq:elliptic_problem} is solved iteratively after
  discretisation of the derivatives by central differences. 

  In the following, we consider a test problem in two dimensions to
  demonstrate the AP property of the scheme. We take the well-prepared
  initial data given in \cite{DegondTang} which reads
  \begin{align}
    \rho(0,x_1,x_2)&=1+\veps^2\sin^2(2\pi(x_1+x_2)), \\
    q_1(0,x_1,x_2)&=\sin(2\pi(x_1-x_2))+\veps^2\sin(2\pi(x_1+x_2)), \\
    q_2(0,x_1,x_2)&=\sin(2\pi(x_1-x_2))+\veps^2\cos(2\pi(x_1+x_2)).   
  \end{align}
  The computational domain $[0,1]\times[0,1]$ is divided into
  $50\times50$ mesh points and we apply periodic boundary conditions
  on all four sides. The CFL number is set to 0.45 and we perform the
  computations up to a final time $T=1.0$. The parameter $\veps$ is
  set to 0.1. Note that our CFL condition is independent of $\veps$.  
  
  In Figures~\ref{fig:LM} and \ref{fig:Bs} we plot the density,
  $x_1$-velocity and the divergence of the velocity at times $t=0$ and
  $t=1$, for the low Mach and Boussinesq cases, respectively. It can be
  noted from the figures that in both the cases the density converges
  to the constant value 1 and the divergence approach 0. This is in
  conformity with the AP nature of the scheme in both the cases.  

  \begin{figure}[htbp]
    \centering
    \includegraphics[height=0.2\textheight]{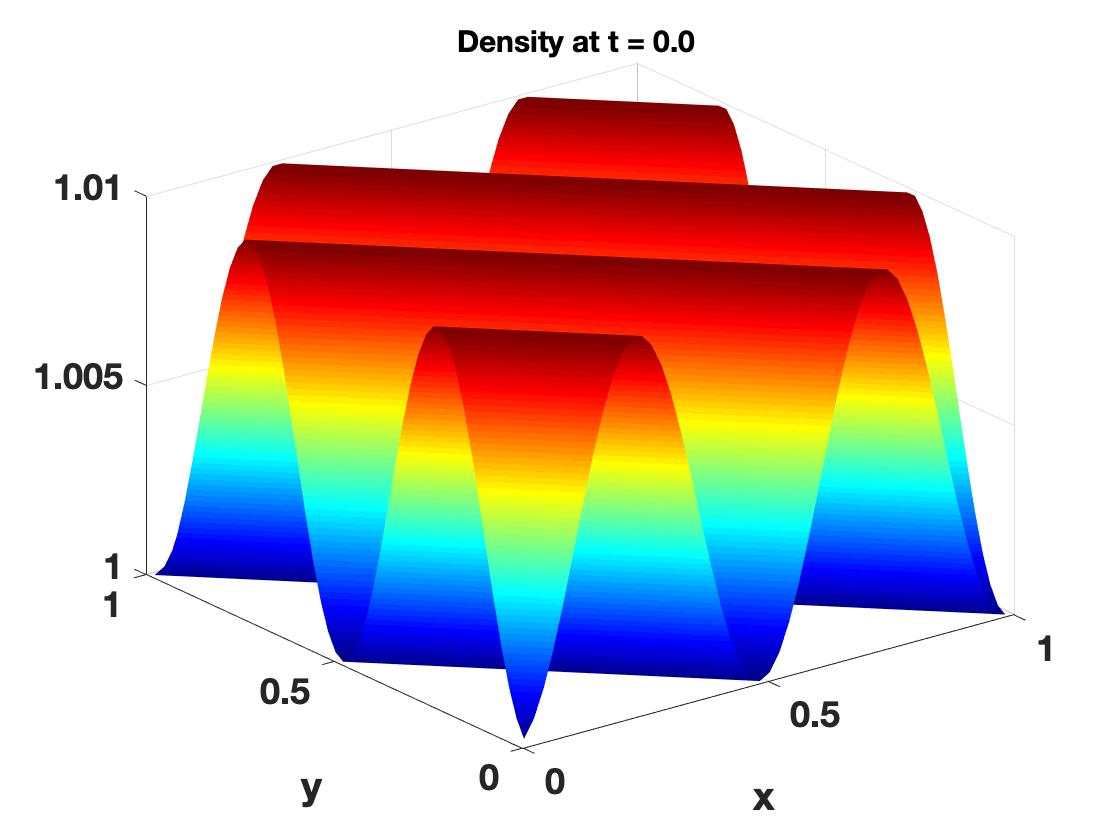} 
    \includegraphics[height=0.2\textheight]{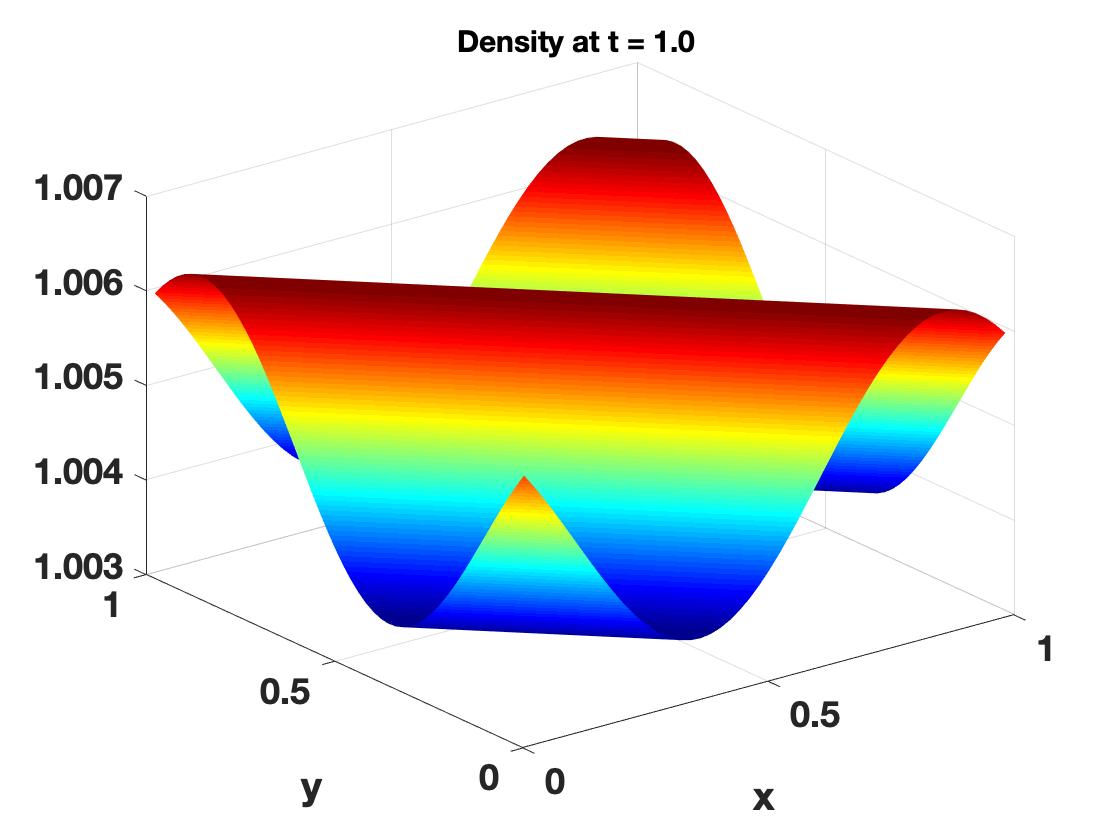}
    \includegraphics[height=0.2\textheight]{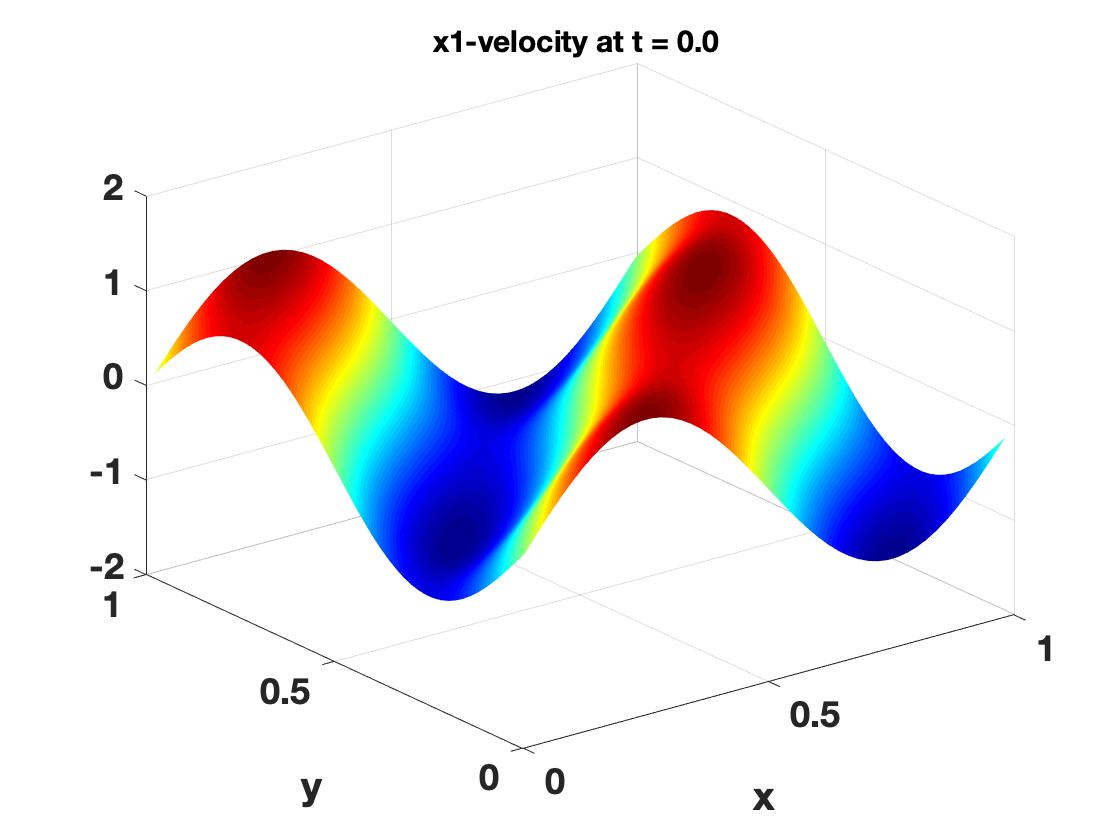}
    \includegraphics[height=0.2\textheight]{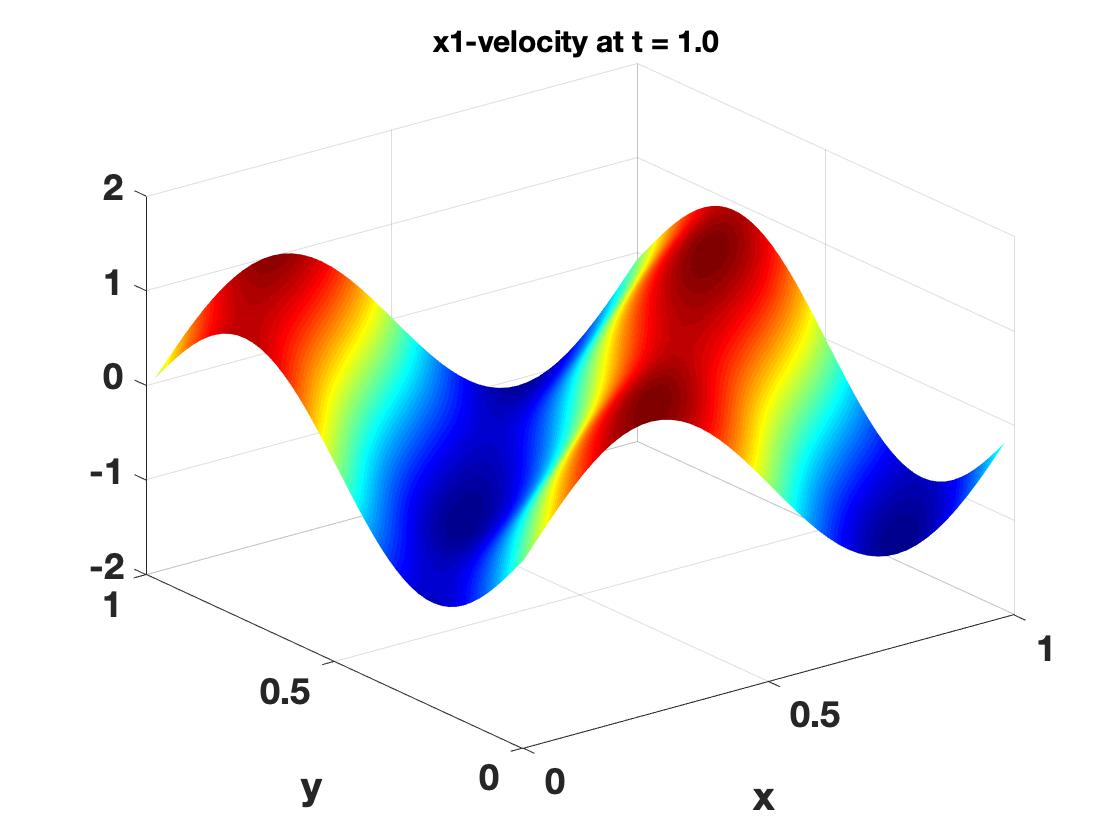}
    \includegraphics[height=0.2\textheight]{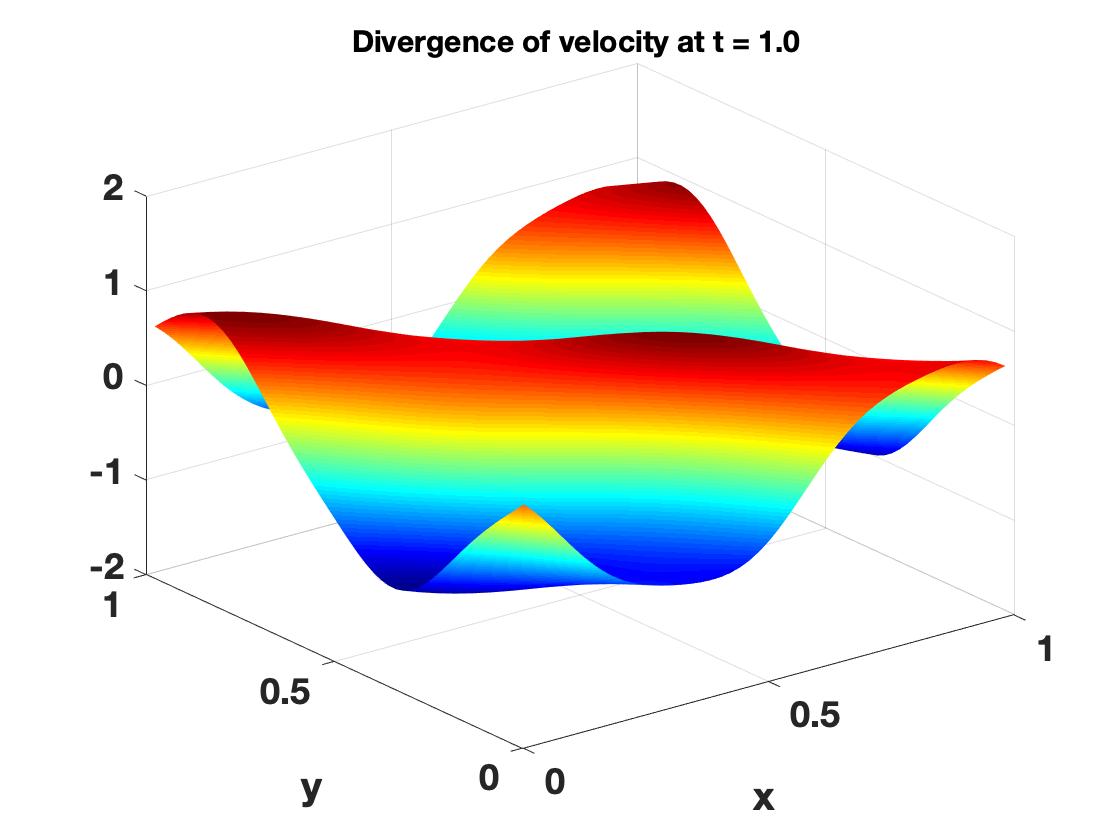} 
    \includegraphics[height=0.2\textheight]{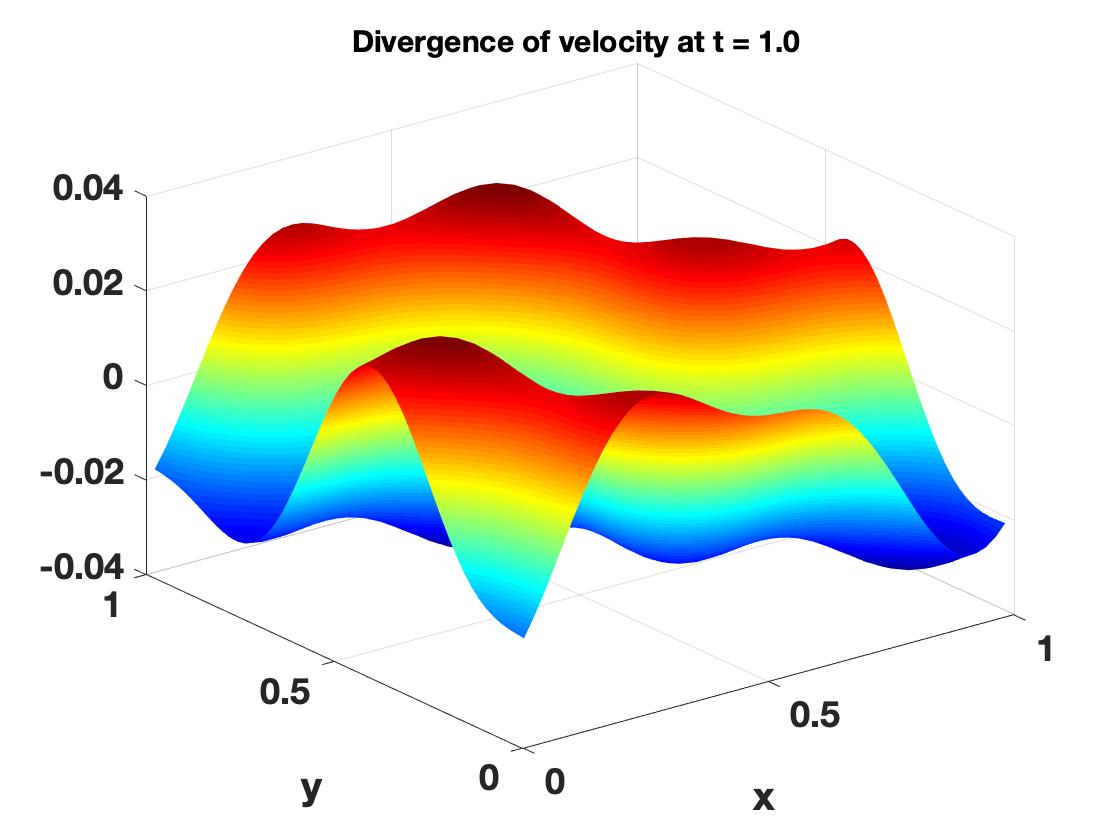} 
    \caption{For $\veps=0.1$, the density, $x_1$-velocity and velocity 
      divergence at $t=0$ (left) the density, $x_1$-velocity and
      velocity divergence at $t=1$. The low Mach number limit. }
    \label{fig:LM}
  \end{figure}

  \begin{figure}[htbp]
    \centering
    \includegraphics[height=0.2\textheight]{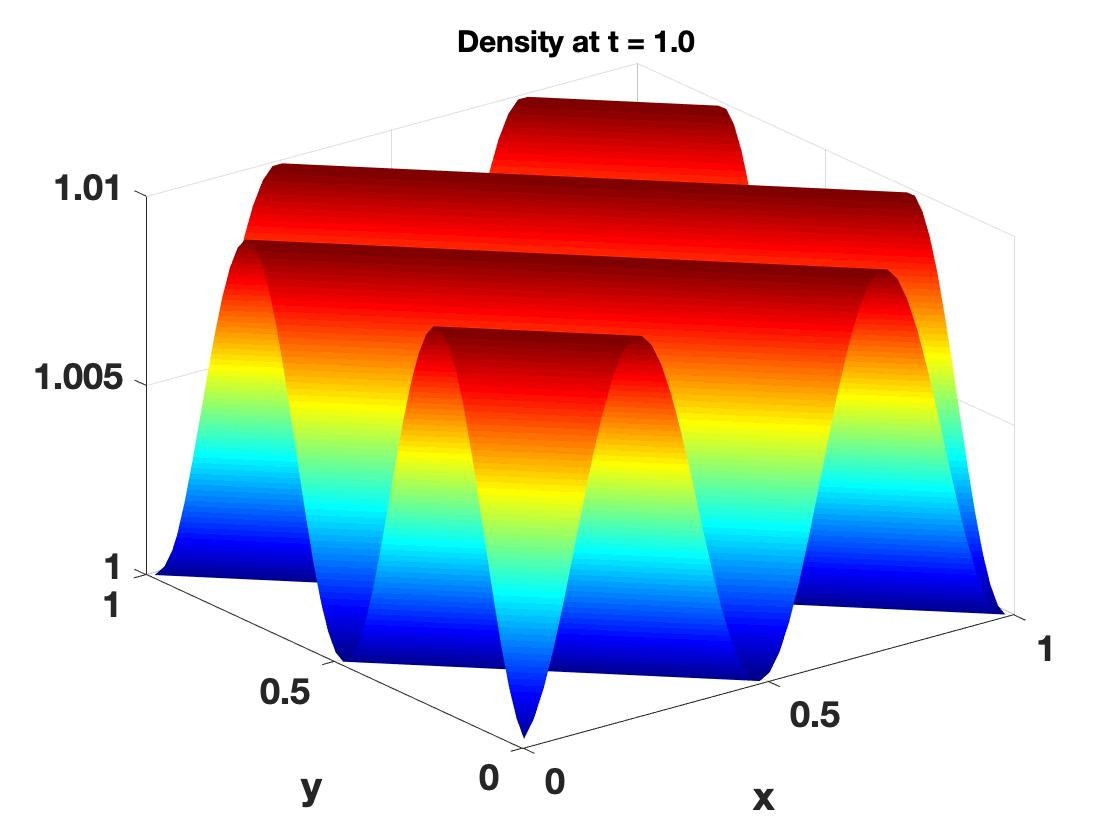} 
    \includegraphics[height=0.2\textheight]{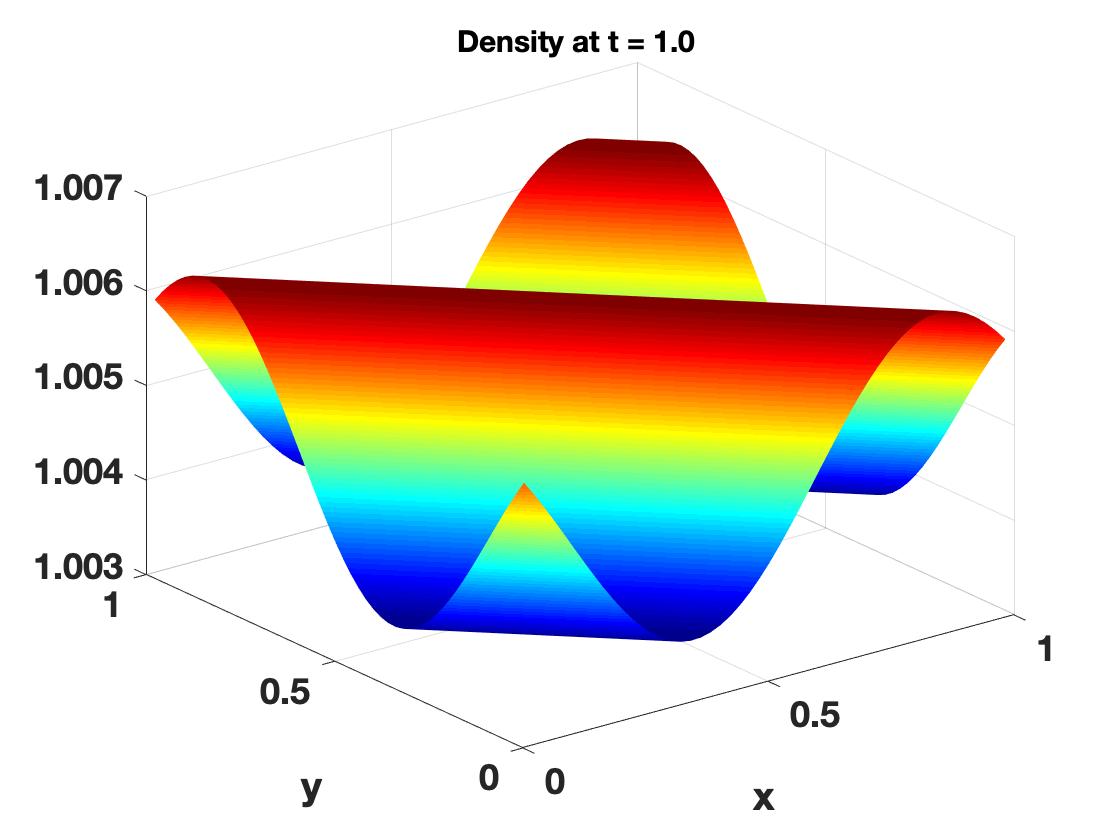}
    \includegraphics[height=0.2\textheight]{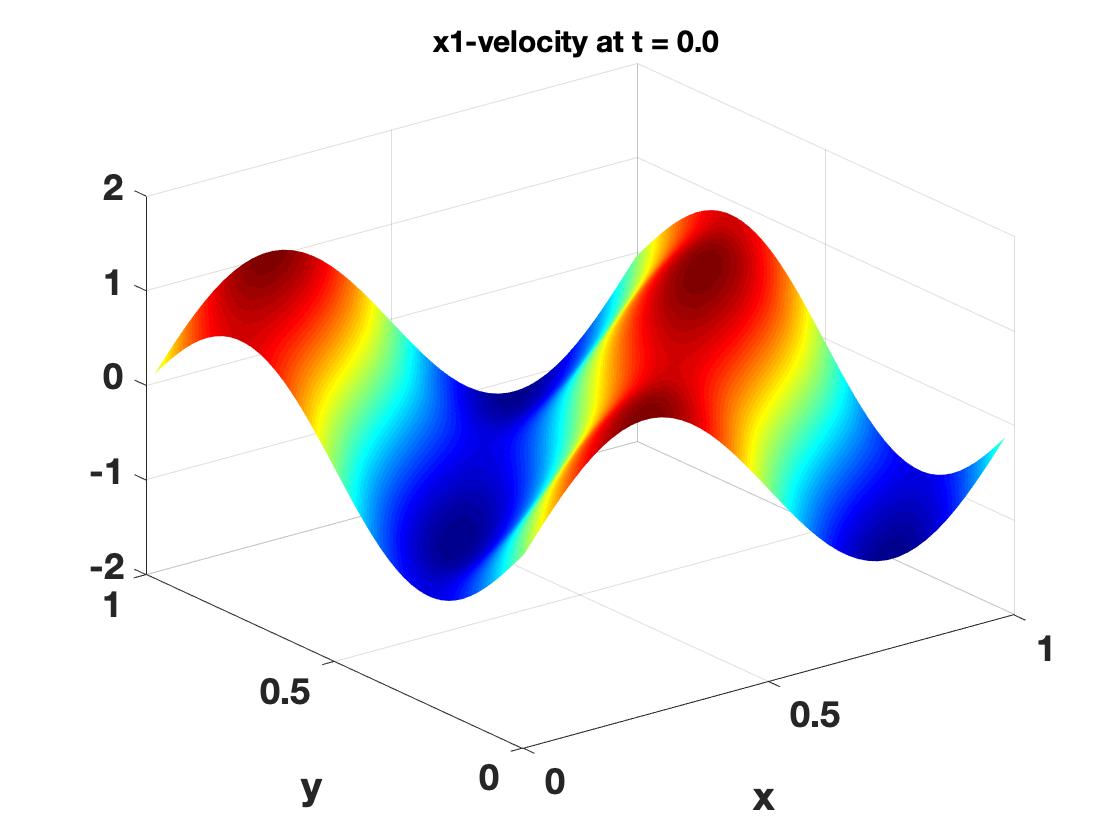}
    \includegraphics[height=0.2\textheight]{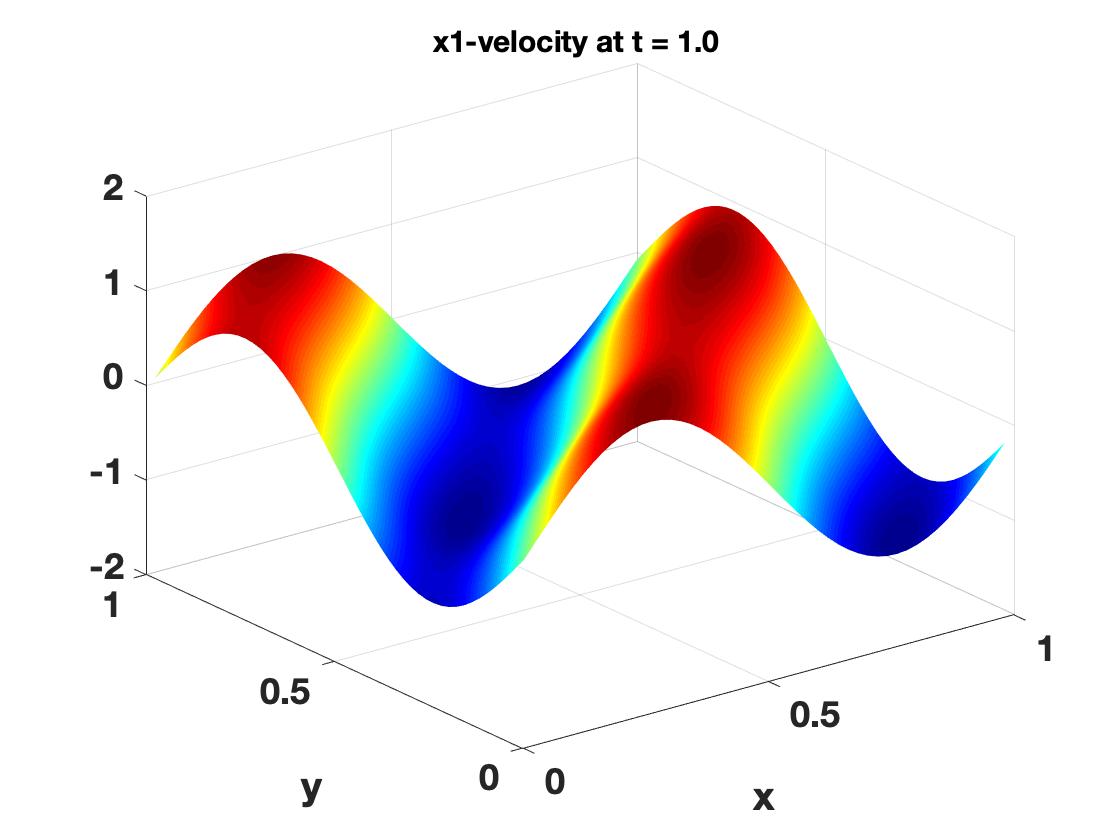}
    \includegraphics[height=0.2\textheight]{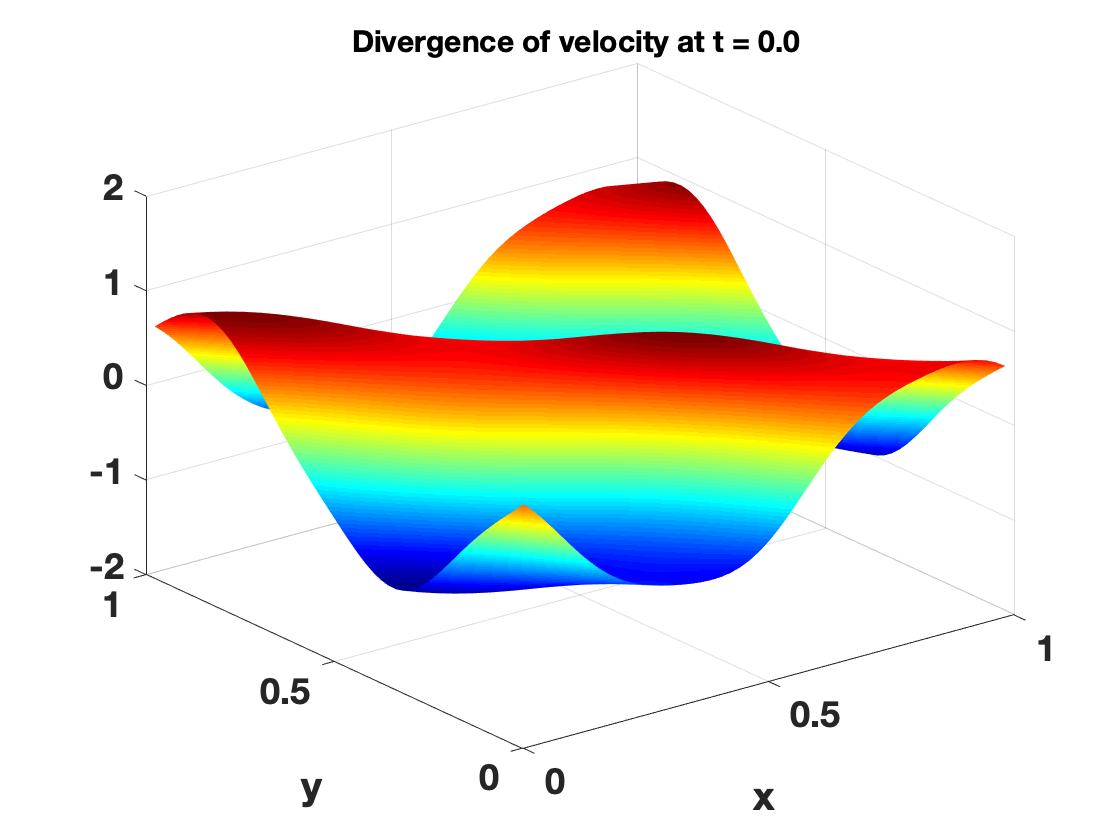} 
    \includegraphics[height=0.2\textheight]{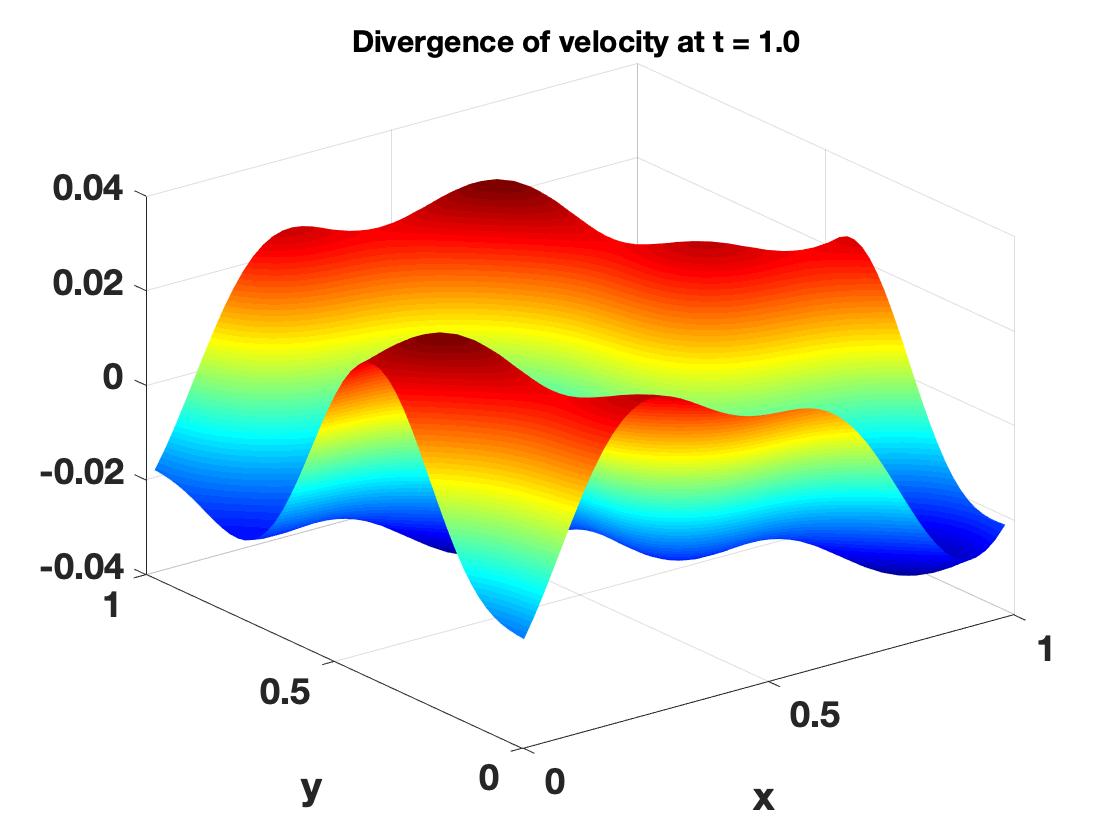} 
    \caption{For $\veps=0.1$, the density, $x_1$-velocity and velocity 
      divergence at $t=0$ (left) the density, $x_1$-velocity and
      velocity divergence at $t=1$. The Boussinesq limit. }
    \label{fig:Bs}
  \end{figure}

  \section{Conclusion}
  An AP semi-implicit time discretisation is proposed for the numerical
  approximation of the isentropic Euler equations with gravity in the
  low Mach number and Boussinesq limits. The schemes are theoretically
  shown to be asymptotically consistent as well as linearly
  stable. The results of numerical experiments provide a justification
  to AP nature of the scheme.

  \section*{Acknowledgement}
  The authors thank Arnab Das Gupta for several useful discussions on
  the topic.

\medskip
Received xxxx 20xx; revised xxxx 20xx.
\medskip

\end{document}